\documentclass[12pt]{article}
\usepackage{amsmath}
\usepackage{amsmath,amssymb,amsthm,graphicx,mathrsfs}

\usepackage{hyperref}

\newcommand{\beq}{\begin{equation}}
\newcommand{\enq}{\end{equation}}

\newtheorem{Theorem}{Theorem}[section]
\newtheorem{Lemma}[Theorem]{Lemma}
\newtheorem{Corollary}[Theorem]{Corollary}
\newtheorem{Definition}[Theorem]{Definition}
\newtheorem{Proposition}[Theorem]{Proposition}

\usepackage[numbers,sort&compress]{natbib}
\newtheorem{Remark}[subsection]{Remark}

\newcommand{\benu}{\begin{enumerate}}
\newcommand{\beqa}{\begin{eqnarray}}
\newcommand{\beqan}{\begin{eqnarray*}}
\newcommand{\eay}{\end{array}}
\newcommand{\edm}{\end{displaymath}}
\newcommand{\eenu}{\end{enumerate}}
\newcommand{\eeq}{\end{equation}}
\newcommand{\eeqa}{\end{eqnarray}}
\newcommand{\eeqan}{\end{eqnarray*}}

\newcommand{\br}{\begin{Remark}}
\newcommand{\er}{\end{Remark}}

\newcommand{\bqa}{\begin{eqnarray}}
\newcommand{\eqa}{\end{eqnarray}}
\newcommand{\bqw}{\begin{eqnarray*}}
\newcommand{\eqw}{\end{eqnarray*}}

\newcommand{\bea}{\begin{array}{cc}}
\newcommand{\ena}{\end{array}}

\def\V{\mathcal{V}}

\def\Z{\mathbb{Z}}

\def\R{\mathcal{R}}

\def\C{\mathbb{C}}

\begin{document}
\pagenumbering{arabic} \setcounter{page}{1}

\begin{center}

{\Large \bf
Simple smooth modules over the Ramond algebra and applications to vertex operator superalgebras}\footnote[2]{ Y.C. is partially supported by the CSC of China (Grant No. 202306630062).  Y.Y. is partially supported by the National Natural Science Foundation of China (12271345 and 12071136).
K.Z. is partially supported by NSERC (311907-2020).}\\
\vspace{0.3in}Yulu Chen$^{1}$,  Yufeng Yao$^{2}$, Kaiming Zhao$^{3, 4}$\\
\vspace{0.3in}

$^{1}$\ School of Mathematics and Statistics, Donghua University, Shanghai 201620, China\\
$^{2}$\ Department of Mathematics, Shanghai Maritime University, Shanghai 201306, China\\
$^{3}$\ Department of Mathematics, Wilfrid Laurier University, Waterloo, Ontario N2L 3C5, Canada\\
$^{4}$\ School of Mathematics and Statistics,
Xinyang Normal University,
Xinyang 464000,   China\\
Email: yulchen@wlu.ca,  yfyao@shmtu.edu.cn, kzhao@wlu.ca
\end{center}
\vspace{3mm}

\begin{abstract}\normalsize
Simple smooth modules over the Virasoro algebra and one of the super-Virasoro algebras, named the Neveu-Schwarz algebra, have been classified.
This problem remained unsolved for the other super-Virasoro algebra called the Ramond algebra.
In this paper, all simple smooth modules over the Ramond algebra are classified.
More precisely, we show that a simple smooth module over the Ramond algebra is either a simple highest weight module or isomorphic to an induced module from a simple module over a finite dimensional solvable Lie superalgebra.
As an application we obtain all simple weak $\psi$-twisted modules over some vertex operator superalgebras.
\end{abstract}
\vspace{3mm}


\


\section{Introduction}

The super-Virasoro algebras are  natural super generalizations of the Virasoro algebra, which are infinite dimensional Lie superalgebras with a long history in mathematical physics. There are two super extensions of the Virasoro algebra with particular importance in the conformal field theory and the superstring theory, called the Ramond algebra\cite{R1971} and Neveu-Schwarz algebra\cite{NS1971}. The even parts of both Lie superalgebras are isomorphic to the Virasoro algebra.


The representation theory of the super-Virasoro algebras has attracted great attention since it plays a fundamental role in superstring theory and conformal theory. Some researchers developed the representation theory from the point of view of vertex algebras (see \cite{A1997, AJR2021, FFF, L, KW,Ba2000, BMRW2017, HM2002, Na2024, CMHY2024}).
The highest weight modules for the Ramond algebra and Neveu-Schwarz algebra were investigated by many physicists and mathematicians (see, e.g. \cite{D2001, D1986, D1987, D1989, GKO1986, MC1986}). All simple unitary weight modules with finite dimensional weight spaces over the super-Virasoro algebras, including highest and lowest weight modules, were classified in \cite{CP1988}. Iohara and Koga described the structure of Verma modules, pre-Verma modules and Fock modules over the  Ramond algebra in \cite{IK2003-1, IK2003-2, IK2006}. All simple Harish-Chandra modules over the  Ramond algebra were claimed firstly by Su in \cite{S1995}, and then Cai et al. confirmed the result by a new approach based on the $A$-cover theory in \cite{CLL2021}.

Recently some non-weight modules for many Lie (super)algebras have    been widely studied. For example, the $U(\mathfrak{h})$-free modules of rank 1 over the Ramond algebra and such modules of rank 2 over Neveu-Schwarz algebra were classified in \cite{YYX2020}.  Chen et al. introduced a new family of functors so that they recovered some old irreducible super-Virasoro modules, including those from the irreducible intermediate series as well as irreducible $U(\mathfrak{h})$-free modules, and provided some non-weight irreducible super-Virasoro modules in \cite{CDLP2024}.
In \cite{LPX2019}, Liu, Pei and Xia obtained the necessary and sufficient conditions for the irreducibility of the Whittaker modules over the super-Virasoro algebras.

Smooth modules are generalizations of both highest weight modules and Whittaker modules. Liu, Pei and Xia \cite{LPX2020} classified simple smooth modules over the  Neveu-Shwarz algebra in the spirit of the work of Mazorchuk and Zhao on the simple Virasoro modules in \cite{MZ2014}.

Some simple smooth modules over the Ramond algebra were constructed by Chen in \cite{C2023},   without  giving the classification of simple smooth Ramond modules. In the present paper we are able to establish new approaches to obtain the classification for all simple smooth modules over the Ramond algebra.

The paper is organized as follows. In section \ref{Sec Pre}, we introduce some basic definitions and   known facts for later use. In section \ref{Sec Construction}, we construct a class of simple induced modules over Ramond algebra including Whittaker modules and the modules constructed in \cite{C2023}, see Theorem \ref{Th simple}. In section \ref{Sec Characterization} we prove that there is a one to one correspondence between such simple smooth Ramond modules and simple modules over a finite dimensional solvable Lie superalgebra. Furthermore, we characterize the simple highest weight Ramond modules by a new method based on the irreducibility of Verma modules over Ramond algebra, see Theorem \ref{h w m}.
And we can prove that a simple smooth module over the  Ramond algebra is either a simple highest weight module or isomorphic to an induced module from a simple module over a finite dimensional solvable Lie superalgebra,  see Theorem \ref{Th equivalent}. Besides, we are able to classify all simple modules over two finite dimensional solvable Lie superalgebras with small dimension related to Ramond algebra.
In section \ref{Sec weakmod}, we apply the results in Section 4 to obtain
all simple weak $\psi$-twisted  modules over some vertex operator superalgebras, see Proposition \ref{5.2}. In section \ref{Sec Examples}, we give some examples which not only recover the highest weight modules and Whittaker modules over the Ramond algebra, but also produce new simple modules.

We denote by $\mathbb{Z}$, $\mathbb{N}$, $\mathbb{Z}_+$ and $\mathbb{C}$ the sets of integers, nonnegative integers, positive integers and complex numbers, respectively. All vector superspaces (resp. superalgebras, supermodules) and spaces (resp. algebras, modules) are considered to be over $\mathbb{C}$.

\section{Preliminaries}\label{Sec Pre}
In this section we collect some related definitions and notations.
Let $V=V_{\overline{0}}\oplus V_{\overline{1}}$ be a $\mathbb{Z}_2$-graded vector space. We say that a vector $v$ is even (resp. odd) if ${v}\in V_{\overline{0}}$ (resp. $v\in V_{\overline{1}}$), and denote its parity as $|v|=\overline{0}$ (resp. $|v|=\overline{1}$). Vectors in $V_{\overline{0}}$ or $V_{\overline{1}}$ are called homogeneous. We make the convention that a vector $v$ is   homogeneous if  $|v|$ is used throughout the paper.

\begin{Definition}\label{Def R}
The  Ramond algebra
$\mathcal{R}=\mathcal{R}_{\overline{0}}\oplus\mathcal{R}_{\overline{1}}$ is the Lie superalgebra
$$\mathcal{R}={\bigoplus\limits_{m\in\mathbb{Z}}{\mathbb{C}L_m}} \oplus {\bigoplus\limits_{m\in\mathbb{Z}}{\mathbb{C}G_m}} \oplus {\mathbb{C}c},$$
where $\mathcal{R}_{\overline{0}}={\rm span}_{\mathbb{C}}\{L_m, c| m\in\mathbb{Z}\}$ and $\mathcal{R}_{\overline{1}}={\rm span}_{\mathbb{C}}\{G_m| m\in\mathbb{Z}\}$,
  subject to the following commutation relations:
\begin{align*}
    &[L_m, L_n]=(m-n)L_{m+n}+\delta_{m+n,0}\frac{m^3-m}{12}c, \\
    &[L_m, G_n]=\left(\frac{m}{2}-n\right)G_{m+n}, \\
    &[G_m, G_n]=2L_{m+n}+\frac{1}{3}\delta_{m+n,0}\left(m^2-\frac{1}{4}\right)c, \\
    &[\mathcal{R}, c]=0,
\end{align*}
for $m, n\in\mathbb{Z}$.
\end{Definition}

It follows from the definition that the even part $\mathcal{R}_{\overline{0}}$ is isomorphic to the classical Virasoro algebra with the center $\mathbb{C}c$. For $m\in \mathbb{Z}$, set
$$\mathcal{R}_m={\rm span}_\mathbb{C}\{L_m, G_m, \delta_{m, 0}c\}.$$
Then $\mathcal{R}$ is a $\mathbb{Z}$-graded algebra with grading determined by the eigenvalues of the adjoint action of $-L_0$,
and $\mathcal{R}$ has the following triangular decomposition
$$\mathcal{R}=\mathcal{R}_{+}\oplus\mathcal{R}_{0}\oplus\mathcal{R}_{-},$$
where
$$\mathcal{R}_{\pm}={\rm span}_{\mathbb{C}}\{L_{\pm m}, G_{\pm m}\ |\ m\in\mathbb{Z}_+\}\ \text{and}\ \mathcal{R}_{0}
={\rm span}_{\mathbb{C}}\{L_0, G_0, c\}.$$
Moreover, for $m, n\in\mathbb{Z}$ we define
$$\mathcal{R}^{(m,n)}:={\bigoplus\limits_{i\geq m}{\mathbb{C}L_i}}
\oplus {\bigoplus\limits_{j\geq n}{\mathbb{C}G_j}} \oplus {\mathbb{C}c}.$$

An $\mathcal{R}$-module is a $\mathbb{Z}_2$-graded vector space $M$ with a bilinear map
$\mathcal{R}\times M\rightarrow M$; $(x, v)\mapsto xv$, satisfying
$$x(yv)-(-1)^{|x||y|}y(xv)=[x, y]v\ \text{and}\ \mathcal{R}_iM_j\subseteq M_{i+j}$$
for any $x, y\in\mathcal{R}$, $v\in M$ and $i, j\in\mathbb{Z}_2$.
Clearly there is a parity-change functor on the category of $\mathcal{R}$-modules
interchanging the $\mathbb{Z}_2$-grading of a module.
Denote by $U(\mathcal{R})$ the universal enveloping algebra of $\mathcal{R}$.

\begin{Definition}
Let $M$ be an $\mathcal{R}$-module.
\begin{enumerate}
  \item[\rm (1)] The action of $x\in\mathcal{R}$ on $M$ is called locally nilpotent if for any $v\in M$,
  there exists $n\in\mathbb{Z}_+$ such that $x^nv=0$.
  The action of $\mathcal{R}$ on $M$ is locally nilpotent if for any $v\in M$,
  there exists $n\in\mathbb{Z}_+$ such that $\mathcal{R}^nv=0$.
  \item[\rm (2)] The action of $x\in\mathcal{R}$ on $M$ is called locally finite if
  ${\rm dim}\left(\sum_{n\in\mathbb{Z}_+}\mathbb{C}x^nv\right)$ $<+\infty$ for any $v\in M$.
      The action of $\mathcal{R}$ on $M$ is locally finite if
      ${\rm dim}\left(\sum_{n\in\mathbb{Z}_+}\mathcal{R}^nv\right)<+\infty$ for any $v\in M$.
\end{enumerate}
\end{Definition}

Obviously, for any Lie (super)algebra $\mathfrak{L}$, if the action of $x\in\mathfrak{L}$ on $M$ is locally nilpotent, then the action of $x$ is locally finite. Also if $\mathfrak{L}$  is locally nilpotent on $M$, then   the action of  $\mathfrak{L}$ is locally finite on $M$ provided  $\mathfrak{L}$ is finitely generated.

\begin{Definition}
    An $\mathcal{R}$-module $M$ is called a smooth module if for any ${v}\in M$, there exists $n\in\mathbb{N}$ such that  $\mathcal{R}_mv=0$ for all $m>n$.
\end{Definition}

\begin{Definition}
We say that an $\mathcal{R}$-module $M$ is of central charge $l$, if $c$ acts on $M$ as a complex scalar $l$.
\end{Definition}

Denote by
$$\mathfrak{b}:=\bigoplus_{m\geqslant0}(\mathbb{C}L_m\oplus\mathbb{C}G_m)\oplus\mathbb{C}c,$$
which is a subalgebra of $\mathcal{R}$. Given a $\mathfrak{b}$-module $V$,
we consider the induced module
$${\rm Ind}(V):={\rm Ind}_{\mathfrak{b}}^{\mathcal{R}}V=U(\mathcal{R})\otimes_{U(\mathfrak{b})}V.$$
If $V$ is a simple $\mathfrak{b}$-module, then $c$ acts on $V$ as a scalar $l\in\mathbb{C}$, and ${\rm Ind}(V)$ has central charge $l$.

For the infinite vectors of the form ${\bf i}:=(\cdots,\ i_2,\ i_1)$, denote by
$$\mathbb{M}:=\{{\bf i}=(\cdots,\ i_2,\ i_1)\mid i_k\in\mathbb{N}\ \text{and the number of nonzero entries is finite}\}$$
and
$$\mathbb{M}_1:=\{{\bf i}\in\mathbb{M}\mid i_k=0, 1\text{ for } k\in\mathbb{Z}_+\}.$$
Let ${\bf 0}=(\cdots,\ 0,\ 0)\in\mathbb{M}$ and $\varepsilon_k=(\cdots,\ 0,\ 1,\ 0,\ \cdots,\ 0)$ for $k\in\mathbb{Z}_+$, where $1$ is in the $k$-th position from the right. For ${\bf i}\neq{\bf 0}$, denote by $\hat{i}$ the minimal integer $k$ such that $i_k\neq0$ and define ${\bf i'}:={\bf i}-\varepsilon_{\hat{i}}$.

For ${\bf k}\in\mathbb{M}_1, {\bf i}\in\mathbb{M}$, let
$$G^{\bf k}L^{\bf i}=\cdots G^{k_2}_{-2}G^{k_1}_{-1}\cdots L^{i_2}_{-2}L^{i_1}_{-1}\in U(\mathcal{R}_-)$$
and
$$w({\bf k},{\bf i})=\sum\limits_{n\in\mathbb{Z}_+}n(k_n+i_n),$$
called the weight of $({\bf k},{\bf i})$. Denote by
$$U(\mathcal{R}_-)_{-m}:=\{G^{\bf k}L^{\bf i}|w({\bf k},{\bf i})=m\},\ m\in\mathbb{Z}_+.$$
According to the PBW Theorem, any vector ${v}\in{\rm Ind}_l(V)$ can be written in the form $$v=\sum\limits_{{\bf k}\in\mathbb{M}_1, {\bf i}\in\mathbb{M}}G^{\bf k}L^{\bf i}v_{{\bf k},{\bf i}},$$
where $v_{{\bf k},{\bf i}}\in V$ and only finitely many of them are nonzero. Then we define the support set of $v$
$${\rm supp}(v):=\{({\bf k},{\bf i})\in\mathbb{M}_1\times\mathbb{M}\mid v_{{\bf k},{\bf i}}\neq0\},$$
and clearly it is finite.

We need the following   total orders.
\begin{Definition}
Denote by $<$ the reverse lexicographic total order on $\mathbb{M}$ (or $\mathbb{M}_1$), i.e., for  any ${\bf i, j}\in\mathbb{M}$ (or in  $\mathbb{M}_1$), we say ${\bf i}<{\bf j}$ if there is $k\in \mathbb{Z}_+$ such that $i_k<j_k$ and $i_s=j_s$ for all $s<k$.
\end{Definition}

\begin{Definition}
Denote by $\prec$ the principle total order on $\mathbb{M}_1\times\mathbb{M}$, i.e., for ${\bf k, m}\in\mathbb{M}_1$ and ${\bf i, j}\in\mathbb{M}$, set $({\bf k},{\bf i})\prec({\bf m}, {\bf j})$ if  one of following is satisfied:
\begin{enumerate}
  \item[\rm(1)] $w({\bf k},{\bf i})<w({\bf m}, {\bf j})$;
  \item[\rm(2)] $w({\bf k},{\bf i})=w({\bf m}, {\bf j})$ and ${\bf k}<{\bf m}$;
  \item[\rm(3)] $w({\bf k},{\bf i})=w({\bf m}, {\bf j})$, ${\bf k}={\bf m}$ and ${\bf i}<{\bf j}$.
\end{enumerate}
\end{Definition}

With respect to the principle total order on $\mathbb{M}_1\times\mathbb{M}$, denote by
$${\rm deg}(v):={\rm max}\{({\bf k},{\bf i})\mid ({\bf k},{\bf i})\in{\rm supp}(v)\}$$
for $v\in {\rm Ind}_l(V)$, called the degree of $v$.

\section{Construction of simple $\mathcal{R}$-modules}\label{Sec Construction}
In this section, we construct some simple smooth $\mathcal{R}$-modules including the ones in \cite{C2023}.
\begin{Lemma}\label{Lem G_jV}
Let $V$ be a $\mathfrak{b}$-module and assume that there exists ${t}\in\mathbb{Z}_+$ such that
\begin{enumerate}
	\item[\rm(1)] the action of $L_t$ on $V$ is injective;
	\item[\rm(2)] $L_iV=0$ for all $i>t$.
\end{enumerate}
Then $G_jV=0$ for all $j>t$.

\end{Lemma}

\begin{proof}
	Since
	\begin{equation*}
		G_{j}V=\frac{2}{j}[L_{j}, G_0]V=\frac{2}{j}(L_{j}G_0-G_0L_{j})V=0
	\end{equation*}
	for any $j>t\geq 1$, the assertion follows.
\end{proof}

\begin{Lemma}\label{Lem (k,i)}
Assume that $V$ is a $\mathfrak{b}$-module (not necessarily simple) satisfying the conditions in Lemma \ref{Lem G_jV}. If $W$ is a nonzero submodule of ${\rm Ind}(V)$, then $W\cap V\neq\{0\}$.
\end{Lemma}

\begin{proof}
To the contrary, suppose that $W\cap V=\{0\}$. Take a nonzero vector $v\in W$ such that ${\rm deg}(v)=({\bf k}, {\bf i})$ is minimal.
By the PBW Theorem, the vector $v$ can be written in the form
$$v=\sum\limits_{({\bf m}, {\bf j})\in{\rm supp}(v)}G^{\bf m}L^{\bf j}v_{{\bf m}, {\bf j}},$$
where $v_{{\bf m}, {\bf j}}\in V$.
Since $W$ is a submodule, we have
\begin{equation*}
  G_tv=\sum\limits_{w({\bf m}, {\bf j})=w({\bf k}, {\bf i})}G^{\bf m}L^{\bf j}G_tv_{{\bf m}, {\bf j}}+\sum\limits_{w({\bf \tilde{m}}, {\bf \tilde{j}})<w({\bf k}, {\bf i})}G^{\bf \tilde{m}}L^{\bf \tilde{j}}v_{{\bf \tilde{m}}, {\bf \tilde{j}}} \in W.
\end{equation*}
If $\sum\limits_{w({\bf m}, {\bf j})=w({\bf k}, {\bf i})}G^{\bf m}L^{\bf j}G_tv_{{\bf m}, {\bf j}}=0$, we obtain $G_tv_{{\bf m}, {\bf j}}=0$ since $G^{\bf m}L^{\bf j}$ are linearly independent for different $({\bf m}, {\bf j})$ according to the PBW Theorem.
Otherwise, $0\neq G_tv\in W$. Since $G_t^2V=0$, we can replace $v$ with $G_tv$ and they share the same degree. Thus we may assume  that $G_tv_{{\bf m}, {\bf j}}=0$ for all $({\bf m}, {\bf j})\in{\rm supp}(v)$ with $w({\bf m}, {\bf j})=w({\bf k}, {\bf i})$.

Let $w({\bf m}, {\bf j})=q\in \mathbb{Z}_+$ for any fixed nonzero $({\bf m}, {\bf j})\in {\rm supp}(v)$. In order to obtain a contradiction, we shall construct a nonzero vector in $W$ with degree lower than ${\rm deg}(v)$ by the following two claims.

{\bf Claim 1.} If ${\bf k}\neq {\bf 0}$, then ${\rm deg}(G_{\hat{k}+t}v)=({\bf k', i})$.

Indeed, note $L_tv_{{\bf m}, {\bf j}}\neq0\ {\rm and}\ G_{\hat{k}+t}v_{{\bf m}, {\bf j}}=0$ due to the Lemma \ref{Lem G_jV}. Thus we have
\begin{equation*}
G_{\hat{k}+t}G^{\bf m}L^{\bf j}v_{{\bf m}, {\bf j}} =[G_{\hat{k}+t},G^{\bf m}]L^{\bf j}v_{{\bf m}, {\bf j}}+G^{\bf m}[G_{\hat{k}+t}, L^{\bf j}]v_{{\bf m}, {\bf j}}.
\end{equation*}
If $\hat{k}+t-q>t$, i.e. $\hat{k}>q$, we see that $G_{\hat{k}+t}G^{\bf m}L^{\bf j}v_{{\bf m}, {\bf j}}=0.$
Next we consider the case  $\hat{k}\leq q\leq w({\bf k}, {\bf i})$. After transferring the factors in $\mathcal{R}_0\oplus\mathcal{R}_+$ to the right side in each term by the commutation relations in Definition \ref{Def R}, the equation above leads to
\begin{equation}\label{equ G GLv}
\begin{aligned}
G_{\hat{k}+t}G^{\bf m}L^{\bf j}v_{{\bf m}, {\bf j}}
\in&\sum\limits_{n=0}^{\hat{k}+t}U(\mathcal{R}_-)_{-q+(\hat{k}+t)-n}\mathcal{R}_{n}v_{{\bf m}, {\bf j}}\\
&=\left(U(\mathcal{R}_-)_{-q+\hat{k}+t}(\mathcal{R}_0+\mathbb{C})
+\cdots+U(\mathcal{R}_-)_{-q+\hat{k}}\mathcal{R}_t\right)v_{{\bf m}, {\bf j}}\\
&=U(\mathcal{R}_-)_{-q+\hat{k}}L_tv_{{\bf {m}}, {\bf {j}}}
+U(\mathcal{R}_-)_{-q+\hat{k}}G_tv_{{\bf {m}}, {\bf {j}}}\\
&\ \ \ +\text{the terms with lower weight}.
\end{aligned}
\end{equation}
If $G_{\hat{k}+t}G^{\bf m}L^{\bf j}v_{{\bf m}, {\bf j}}\neq 0$,  denote ${\rm deg}(G_{\hat{k}+t}G^{\bf m}L^{\bf j}v_{{\bf m}, {\bf j}})=({\bf m_1}, {\bf j_1})\in\mathbb{M}_1\times\mathbb{M}$. Then $w({\bf m_1}, {\bf j_1})\leq q-\hat{k}$  which directly follows from (\ref{equ G GLv}).

{\bf Case 1.1: $w({\bf m}, {\bf j})<w({\bf k}, {\bf i})$.}

We deduce that
$$w({\bf m_1}, {\bf j_1})\leq q-\hat{k}<w({\bf k}, {\bf i})-\hat{k}=w({\bf k'}, {\bf i}).$$
Thus $({\bf m_1}, {\bf j_1})\prec({\bf k'}, {\bf i})$.

{\bf Case 1.2: $w({\bf m}, {\bf j})=w({\bf k}, {\bf i})$ and ${\bf m}<{\bf k}$.}

Note that $G_tv_{{\bf m}, {\bf j}}=0$ in this case. It follows from ${\bf m}<{\bf k}$ that ${\bf m}={\bf 0}$ or $\hat{m}\geq\hat{k}>0$.
We see that $L_t$ occurs in (\ref{equ G GLv}) if and only if $m_{\hat{k}}\neq 0$.

If ${\bf m}={\bf 0}$ or $\hat{m}>\hat{k}>0$, then the first two summands in (\ref{equ G GLv}) vanish. Thus we obtain
$$w({\bf m_1}, {\bf j_1})< q-\hat{k}=w({\bf k}, {\bf i})-\hat{k}=w({\bf k'}, {\bf i}),$$
yielding $({\bf m_1}, {\bf j_1})\prec({\bf k'}, {\bf i})$.

If $\hat{m}=\hat{k}>0$, then ${\bf m_1}={\bf m'}<{\bf k'}$ since $L_{t}$ occurs only in the bracket $[G_{\hat{k}+t}, G_{-\hat{k}}]$. And we know
$$w({\bf m_1}, {\bf j_1})=q-\hat{m}=w({\bf k}, {\bf i})-\hat{k}=w({\bf k'}, {\bf i}).$$
Hence, $({\bf m_1}, {\bf j_1})= ({\bf m'}, {\bf j_1})\prec({\bf k'}, {\bf i})$.

{\bf Case 1.3. $w({\bf m}, {\bf j})=w({\bf k}, {\bf i})$ and ${\bf m}={\bf k}$.}

Similar arguments to above yield
$$w({\bf m_1}, {\bf j_1})=w({\bf k'}, {\bf i})=q-\hat{k}\geq0\ {\text{and}}\ {\bf m_1}={\bf m'}={\bf k'}.$$
If ${\bf j}<{\bf i}$, then ${\bf j_1}={\bf j}<{\bf i}$, which reveals that $({\bf m_1}, {\bf j_1})=({\bf m'}, {\bf j})\prec({\bf k'}, {\bf i})$.

In the following we show that $G_{\hat{k}+t}v\neq 0$.
If $w({\bf k}, {\bf i})=q=\hat{k}$, then $\bf j=\bf i=\bf 0$, so we have
$$G_{\hat{k}+t}G^{\bf m}L^{\bf j}v_{{\bf m}, {\bf j}}=G_{\hat{k}+t}G_{-\hat{k}}v_{{\bf m}, {\bf 0}}=2L_tv_{{\bf m}, {\bf 0}},$$
which is a nonzero element. This together with Cases 1.1 and 1.2 yields that  $G_{\hat{k}+t}v\neq 0$, and ${\rm deg}(G_{\hat{k}+t}v)=(\bf{0}, \bf{0})=({\bf k'}, {\bf i})$. If $w({\bf k}, {\bf i})=q>\hat{k}$, then $w({\bf m_1}, {\bf j_1})=w({\bf k'}, {\bf i})>0$, which implies $G_{\hat{k}+t}G^{\bf m}L^{\bf j}v_{{\bf m}, {\bf j}}=2G^{\bf m^{\prime}}L^{\bf j}L_tv_{{\bf m,j}} +\text{the terms with lower weight}\neq 0$ for all $(\bf m, j)$ with $w({\bf m}, {\bf j})=w({\bf k}, {\bf i})$ and ${\bf m}={\bf k}$. In particular, $G_{\hat{k}+t}G^{\bf k}L^{\bf i}v_{{\bf k}, {\bf i}}\neq 0$, so that $G_{\hat{k}+t}v\neq 0$, and ${\rm deg}(G_{\hat{k}+t}v)={\rm deg}(G_{\hat{k}+t}G^{\bf k}L^{\bf i}v_{{\bf k}, {\bf i}})=({\bf k'}, {\bf i})$.

Thus we conclude that $G_{\hat{k}+t}v\neq 0$, and
$${\rm deg}(G_{\hat{k}+t}G^{\bf m}L^{\bf j}v_{{\bf m}, {\bf j}})=({\bf k'}, {\bf i})$$
if and only if $({\bf m}, {\bf j})=({\bf k}, {\bf i})$. Claim 1 follows.

{\bf Claim 2.} If ${\bf k}={\bf 0}$, then ${\bf i}\neq {\bf 0}$ and ${\rm deg}(L_{\hat{i}+t}v)=({\bf 0, i'})$.

Indeed, since $v\notin V$, clearly we see that ${\bf i}\neq {\bf 0}$.
We write $v$ as
$$v=L^{\bf i}v_{{\bf 0}, {\bf i}}+\sum\limits_{({\bf m}, {\bf j})\prec({\bf 0}, {\bf i})}G^{\bf m}L^{\bf j}v_{{\bf m}, {\bf j}}.$$
If $\hat{i}+t-q>t$, i.e. $\hat{i}>q$, we see that $L_{\hat{i}+t}G^{\bf m}L^{\bf j}v_{{\bf m}, {\bf j}}=0.$  Next we consider the case  $\hat{i}\leq q\leq w({\bf 0}, {\bf i})$.
We know that
$$L_{\hat{i}+t}G^{\bf m}L^{\bf j}v_{{\bf m}, {\bf j}}
=[L_{\hat{i}+t}, G^{\bf m}]L^{\bf j}v_{{\bf m}, {\bf j}}+G^{\bf m}[L_{\hat{i}+t},L^{\bf j}]v_{{\bf m}, {\bf j}}.$$
Then similar to (\ref{equ G GLv}), the equation above leads to
\begin{equation}\label{equ L GLv}
\begin{aligned}
L_{\hat{i}+t}G^{\bf m}L^{\bf j}v_{{\bf m}, {\bf j}}
\in&\sum\limits_{n=0}^{t}U(\mathcal{R}_-)_{-q+(\hat{i}+t)-n}\mathcal{R}_{n}v_{{\bf m}, {\bf j}}\\
&=U(\mathcal{R}_-)_{-q+\hat{i}}L_tv_{{\bf m}, {\bf j}}
+U(\mathcal{R}_-)_{-q+\hat{i}}G_tv_{{\bf m}, {\bf j}}\\
&\ \ \ +\text{the terms with lower weight}.\\
\end{aligned}
\end{equation}
If $L_{\hat{i}+t}G^{\bf m}L^{\bf j}v_{{\bf m}, {\bf j}}\neq 0$, denote ${\rm deg}(L_{\hat{i}+t}G^{\bf m}L^{\bf j}v_{{\bf m}, {\bf j}})=({\bf m_2}, {\bf j_2})\in\mathbb{M}_1\times\mathbb{M}$. Then $w({\bf m_2}, {\bf j_2})\leq q-\hat{i}$ by (\ref{equ L GLv}).

If ${\bf m\neq 0}$, we see that $w({\bf m}, {\bf j})<w({\bf 0}, {\bf i})$, yielding $({\bf m_2}, {\bf j_2})\prec({\bf 0}, {\bf i'})$ similar to Case 1.1. So we only need to consider the case ${\bf m=0}$. Note that there are no odd elements in $[L_{\hat{i}+t},L^{\bf j}]v_{{\bf 0}, {\bf j}}$, which implies that ${\bf m_2=0}$.

{\bf Case 2.1: $w({\bf 0}, {\bf j})<w({\bf 0}, {\bf i})$.}

Similar to Case 1.1, we obtain $({\bf 0}, {\bf j_2})\prec({\bf 0}, {\bf i'})$ since
$$w({\bf 0}, {\bf j_2})\leq q-\hat{i}<w({\bf 0}, {\bf i})-\hat{i}=w({\bf 0}, {\bf i'}).$$

{\bf Case 2.2: $w({\bf 0}, {\bf j})=w({\bf 0}, {\bf i})$ and ${\bf j}<{\bf i}$.}

It follows from ${\bf j}<{\bf i}$ that ${\bf j}={\bf 0}$ or $\hat{j}\geq\hat{i}>0$.
Note that there is no $G_t$ in this case, and $L_t$ occurs in (\ref{equ L GLv}) if and only if $j_{\hat{i}}\neq 0$.

If ${\bf j}={\bf 0}$ or $\hat{j}>\hat{i}>0$, then the first two summands in (\ref{equ L GLv}) vanish. Thus
$$w({\bf 0}, {\bf j_2})< q-\hat{i}=w({\bf 0}, {\bf i})-\hat{i}=w({\bf 0}, {\bf i'}),$$
yielding $({\bf 0}, {\bf j_2})\prec({\bf 0}, {\bf i'})$.

If $\hat{j}=\hat{i}>0$, then ${\bf j_2}={\bf j'}<{\bf i'}$ since $L_{t}$ occurs only in the bracket $[L_{\hat{i}+t}, L_{-\hat{i}}]$. And we know
$$w({\bf 0}, {\bf j_2})=q-\hat{j}=w({\bf 0}, {\bf i})-\hat{i}=w({\bf 0}, {\bf i'}).$$
Hence, $({\bf m_2}, {\bf j_2})= ({\bf 0}, {\bf j'})\prec({\bf 0}, {\bf i'})$.

{\bf Case 2.3: $({\bf 0}, {\bf j})=({\bf 0}, {\bf i})$.}

Similar arguments to above yield
$$w({\bf 0}, {\bf j_2})=w({\bf 0}, {\bf i'})=q-\hat{i}\geq0\ {\text{and}}\ {\bf j_2}={\bf j'}={\bf i'}.$$
We remain to show $L_{\hat{i}+t}v\neq 0$.
If $w({\bf 0}, {\bf i})=q=\hat{i}$, then we have
$$L_{\hat{i}+t}L^{\bf i}v_{{\bf 0}, {\bf i}}=L_{\hat{i}+t}L_{-\hat{i}}v_{{\bf 0}, {\bf i}}=(2\hat{i}+t)L_tv_{{\bf 0}, {\bf i}},$$
which is a nonzero element. If $w({\bf 0}, {\bf i})=q>\hat{i}$, then $w({\bf0}, {\bf i'})>0$, which implies $L_{\hat{i}+t}L^{\bf i}v_{{\bf 0}, {\bf i}}\neq 0$. Thus $L_{\hat{i}+t}v\neq 0$.

Hence we conclude that $$L_{\hat{i}+t}v\neq 0\ {\rm and}\ {\rm deg}(L_{\hat{i}+t}L^{\bf j}v_{{\bf 0}, {\bf j}})\preceq({\bf 0}, {\bf i'}).$$
The equality holds if and only if $({\bf 0}, {\bf j})=({\bf 0}, {\bf i})$. Claim 2 follows.

Summing up, we can always get a nonzero vector in $W$ with degree lower than ${\rm deg}(v)$, which contradicts the choice of $v$.
\end{proof}

\begin{Theorem}\label{Th simple}
Let $V$ be a simple  $\mathfrak{b}$-module and assume that there exists ${t}\in\mathbb{Z}_+$ such that
	\begin{enumerate}
		\item[\rm(1)] the action of $L_t$ on $V$ is injective;
		\item[\rm(2)] $L_iV=0$ for all $i>t$.
	\end{enumerate}
Then the induced module ${\rm Ind}(V)$ is a simple smooth $\mathcal{R}$-module.
\end{Theorem}

\begin{proof} The result in this theorem directly follows from Lemma \ref{Lem (k,i)}.
\end{proof}

\section{Simple smooth $\mathcal{R}$-modules}\label{Sec Characterization}
In this section, we give a classification of all simple smooth $\mathcal{R}$-modules. As a first step, we need the following key result.

\begin{Theorem}\label{h w m}
Let $V$ be a simple $\mathcal{R}$-module. If there exists a nonzero vector $v\in V$ such that $\mathcal{R}_+v=0$, then $V$ is a highest weight $\mathcal{R}$-module.
\end{Theorem}

\begin{proof} Since $V$ is a simple $\mathcal{R}$-module, we may assume that  $c$ acts on $V$  as a scalar $l\in\mathbb{C}$.
Let
$$M=U(\mathfrak{b})v=\mathbb{C}[L_0]v\oplus\mathbb{C}[L_0]G_0v,$$
which is  a $\mathfrak{b}$-module. Then we have
$$V=U(\mathcal{R}_-)M.$$
If $M$ is not a free $\mathbb{C}[L_0, G_0]$-module, then $M$ has an eigenvector with respect to $L_0$ which implies that $V$ is a highest weight module. Now suppose that $M$ is a free $\mathbb{C}[L_0, G_0]$-module. Consider the module epimorphism
$$\varphi: {\rm Ind}_{\mathfrak{b}}^{\mathcal{R}}M\rightarrow V.$$
It is sufficient to show that $\varphi$ is an isomorphism, i. e., ${\rm Ker}(\varphi)=\{0\}$.

Clearly, ${\rm Ker}(\varphi)\cap M=\{0\}$. If ${\rm Ker}(\varphi)\neq\{0\}$, we can take a nonzero homogeneous element $u\in {\rm Ker}(\varphi)$. Write
$$u=\sum\limits_{i=1}^{n}u_iv,$$
where
$$u_i=\sum_{j}a_{i, j}f_{i, j}(L_0, G_0)\in U(\mathcal{R}_-\oplus\mathcal{R}_0), i=1, \cdots, n,$$
are nonzero eigenvectors of ${\rm ad}(-L_0)$ with pairwise distinct eigenvalues $-z_i\in -\mathbb{Z}_+$, for $a_{i, j}\in U(\mathcal{R_-})_{-z_i}$ and $f_{i, j}(L_0, G_0)\in\mathbb{C}[L_0, G_0]$.
Since $u\notin M$, we assume that $u_1v\notin M$ and $u_1$ has the minimal eigenvalue $-\alpha$. Let $M(\lambda)$ be a $\mathbb{C}[L_0, G_0]$-submodule of $M$ generated by $L_0v-\lambda v$ for $\lambda\in\mathbb{C}$. Then $M/M(\lambda)$ is a two-dimensional $\mathfrak{b}$-module with
\begin{equation*}
  \begin{aligned}
    &\mathcal{R}_+(v+M(\lambda))=\mathcal{R}_+(G_0v+M(\lambda))=0,\\
    &L_0(v+M(\lambda))=\lambda v+M(\lambda),\ L_0(G_0v+M(\lambda))=\lambda G_0v+M(\lambda), \\
    &G_0(v+M(\lambda))=G_0v+M(\lambda),\ G_0(G_0v+M(\lambda))=(\lambda-\frac{1}{24}l)v+M(\lambda).\\
  \end{aligned}
\end{equation*}

From \cite{IK2003-1} and \cite{IK2006}, it is known that there are infinitely many nonzero $\mu\in\mathbb{C}$ with $\mu\neq \frac{1}{24}l$ such that the Verma module
$$W={\rm Ind}_{\mathfrak{b}}^{\mathcal{R}}(M/M(\mu))$$
is irreducible. Take such a nonzero $\mu_0$ with $f_{1, j}(L_0, G_0)v\neq 0$ in $M/M(\mu_0)$ for some $j$. And we can find an element $x\in U(\mathcal{R}_+)$ with the eigenvalue $\alpha$ with respect to ${\rm ad}(-L_0)$ such that
$$0\neq xu\in \mathbb{C}v$$
in $W_0={\rm Ind}_{\mathfrak{b}}^{\mathcal{R}}(M/M(\mu_0))$. Hence
$$0\neq xu=f_1(L_0)v$$
in ${\rm Ind}_{\mathfrak{b}}^{\mathcal{R}}M,$ where $0\neq f_1(L_0)\in \mathbb{C}[L_0]$ with $f_1(\mu_0)\neq 0$. Thus we get a nonzero element $xu\in M\cap {\rm Ker}(\varphi)$, which is a contradiction.
\end{proof}

\begin{Lemma}\label{Lem v smooth}
Let $V$ be a simple $\mathcal{R}$-module. If there exists a nonzero vector $v\in V$ and $M\in \mathbb{Z}_+$ such that $L_mv=G_mv=0$ for all $m\geq M$, then $V$ is a smooth module.
\end{Lemma}

\begin{proof}
By the PBW theorem and simplicity of $V$, any vector of $V$ can be written as a linear combination of elements of the form
$$g_{i_1}g_{i_2}\cdots g_{i_k}v,$$
where $g_{i_j}\in\mathcal{R}_{i_j}$ for all $i_j\in\mathbb{Z}$ and $i_j<M$. Note that $[\mathcal{R}_i, \mathcal{R}_j]\subseteq\mathcal{R}_{i+j}$. For any $g_{i_1}g_{i_2}\cdots g_{i_k}v\in V$, take $N=M+|i_1|+\cdots+|i_k|$, then it is easy to check $$\mathcal{R}_n(g_{i_1}g_{i_2}\cdots g_{i_k}v)=0,$$
for all $n\geq N$. Thus $V$ is a smooth $\mathcal{R}$-module.
\end{proof}

For $t\in\mathbb{N}$, let $\mathfrak{m}^{(t)}=\bigoplus\limits_{m>t}(\mathbb{C}L_m\oplus\mathbb{C}G_m)$ and $\mathfrak{b}^{(t)}=\mathfrak{b}/\mathfrak{m}^{(t)}$. Note that  $\mathfrak{m}^{(0)}=\mathcal{R}_+$. Now we present our main result.

\begin{Theorem}\label{Th equivalent}
Let $S$ be a simple $\mathcal{R}$-module, then the following statements are equivalent:
\begin{enumerate}
    \item[\rm (1)] there exists ${t}\in\mathbb{Z}_+$ such that the action of $L_t$ on $S$ is locally finite;
    \item[\rm (2)] $S$ is a smooth module;
    \item[\rm (3)] $S$ is either a highest weight module, or isomorphic to the induced module ${\rm Ind}(V)$, where $V$ is a simple $\mathfrak{b}$-module satisfying the conditions in Lemma \ref{Lem G_jV}, that is, $V$ is a simple $\mathfrak{b}^{(t)}$-module for some $t\in\mathbb{Z}_+$.
\end{enumerate}
\end{Theorem}

\begin{proof}
It is enough to prove (1)$\Rightarrow$(2) and (2)$\Rightarrow$(3), since (3)$\Rightarrow$(1) is obvious by definition.

{\bf (1)$\Rightarrow$(2).}
Since $L_t$ acts locally finitely on $S$, there exists a nonzero vector $v\in S$ and $\lambda\in\mathbb{C}$ such that
$$L_tv=\lambda v.$$
Now by the Lemma \ref{Lem v smooth}, our task is to find some $M\in \mathbb{Z}_+$ satisfying $\mathcal{R}_mv=0$ for all $m\geq M$.

Note that there exists $N\in \mathbb{Z}_+$ such that $L_nv=0$ for all $n\geq N$ thanks to the Lemma 3.10 in \cite{MNTZ2023}.
Fix any $j\in\{t+1, t+2, \cdots, 2t\}$. Since ${\rm dim}(U(L_t)G_jv)$ is finite and
$$(L_t-\lambda)G_jv=(\frac{t}{2}-j)G_{t+j}v,$$
there is a smallest $n_j\in\mathbb{N}$ such that
$$G_{j+s_jt}v,\ G_{j+(s_j+1)t}v,\ \cdots,\ G_{j+(s_j+n_j)t}v$$
are linearly dependent for some $s_j\in\mathbb{N}$. In other words, there exists a polynomial
$$p_j(y)=a_{j, 0}+a_{j, 1}(y-\lambda)+\cdots+a_{j, n_j}(y-\lambda)^{n_j},\ a_{j, 0}a_{j, n_j}\neq 0,$$
such that
$$p_j(L_t)G_{j+s_jt}v=0.$$
Applying $L_t-\lambda$ to the equation above repeatedly, we deduce that
\begin{equation}\label{equ p_jG}
p_j(L_t)G_{j+st}v=0
\end{equation}
for all $s\geq s_j$. That is, there exists a polynomial of $L_t$ with the smallest degree $n_j\in\mathbb{N}$ which annihilates $G_{j+st}v$ for sufficiently large $s$.

Assume that we can take $j\in\{t+1, t+2, \cdots, 2t\}$ such that $n_j\geq 1$. Consider (\ref{equ p_jG}) for $2s$, where $st\geq{\rm max}\{N, s_jt\}$, which leads to
\begin{equation}\label{equ q1}
\begin{aligned}
0
&=p_j(L_t)G_{j+2st}v \\
&=\left(a_{j, 0}+a_{j, 1}(L_t-\lambda)+\cdots+a_{j, n_j}(L_t-\lambda)^{n_j}\right)G_{j+2st}v  \\
&=(a_{j, 0}G_{j+2st}+a_{j, 1}(\frac{t}{2}-j-2st)G_{j+(2s+1)t}+\cdots  \\
&+a_{j, n_j}(\frac{t}{2}-j-2st)(\frac{t}{2}-j-(2s+1)t)\cdots(\frac{t}{2}-j-(2s+n_j-1)t)G_{j+(2s+n_j)t})v.
\end{aligned}
\end{equation}
Besides, we have following bracket relation for $st\geq{\rm max}\{N, s_jt\}$
\begin{equation}\label{equ q2}
\begin{aligned}
0
&=\left[L_{st}, p_j(L_t)G_{j+st}\right]v \\
&=\Big[L_{st}, a_{j,0}G_{j+st}+\cdots+a_{j,n_j}(\frac{t}{2}-j-st)\cdots(\frac{t}{2}-j-(s+n_j-1)t)G_{j+(s+n_j)t}\Big]v\\
&=(a_{j, 0}(-\frac{st}{2}-j)G_{j+2st}+a_{j, 1}(\frac{t}{2}-j-st)(\frac{st}{2}-j-(s+1)t)G_{j+(2s+1)t}+\cdots  \\
&+a_{j, n_j}(\frac{t}{2}-j-st)\cdots(\frac{t}{2}-j-(s+n_j-1)t)(\frac{st}{2}-j-(s+n_j)t)G_{j+(2s+n_j)t})v.
\end{aligned}
\end{equation}
Consider the last terms of (\ref{equ q1}) and (\ref{equ q2}), and denote by
\begin{equation*}
\begin{aligned}
q_{j,1}(y)&=a_{j, n_j}(\frac{t}{2}-j-2yt)(\frac{t}{2}-j-(2y+1)t)\cdots(\frac{t}{2}-j-(2y+n_j-1)t), \\
q_{j,2}(y)&=a_{j, n_j}(\frac{t}{2}-j-yt)\cdots(\frac{t}{2}-j-(y+n_j-1)t)(\frac{yt}{2}-j-(y+n_j)t).
\end{aligned}
\end{equation*}
We claim that $(-\frac{yt}{2}-j)q_{j,1}$ and $q_{j,2}$ are linearly independent since
$$q_{j,1}\left(-\frac{2(j+n_jt)}{t}\right)\neq 0\ \text{and}\ q_{j,2}\left(-\frac{2(j+n_jt)}{t}\right)= 0.$$
Thus we can find a nonzero polynomial $q_{j}(y)$ of degree less that $n_j$ such that
$$q_{j}(L_t)G_{j+s't}v=0$$
for sufficiently large $s'$, which is a contradiction to the minimality of $n_j$. Hence $n_j=0$ for all $j\in\{t+1, t+2, \cdots, 2t\}$, and
$$M={\rm max}\{N, j+s_jt|t+1\leq j\leq 2t\},$$
as desired.

{\bf (2)$\Rightarrow$(3).}
Suppose that $S$ is a simple smooth $\mathcal{R}$-module. Then the vector space
$$N_{s}:=\{v\in S|G_kv=L_kv=0\ \text{for all }\ k>s \}$$
is nonzero for sufficiently large $s\in\mathbb{N}$. Denote by $s_0$ the smallest integer such that $N_{s_0}\neq \{0\}$. If $s_0=0$, then $S$ is a highest weight module by Theorem \ref{h w m}.

If $s_0\in\mathbb{Z}_+$, let $V:=N_{s_0}$ for convenience. Then for $k>s_0$ and $i\in\mathbb{N}$, we have
\begin{equation*}
\begin{aligned}
&L_k(G_iv)=(\frac{k}{2}-i)G_{k+i}v=0, \\
&G_k(G_iv)=2L_{k+i}v=0,
\end{aligned}
\end{equation*}
where $v\in V$, yielding $G_iv\in V$ for all $i\in\mathbb{N}$. Similarly, $L_iv\in V$ for all $i\in\mathbb{N}$. Hence $V$ is a $\mathfrak{b}$-module. We claim the action of $L_{s_0}$ on $V$ is injective. Otherwise, let
$${\rm Ann}_{V}(L_{s_0})=\{v\in V| L_{s_0}v=0\}\neq \{0\},$$
on which the action of $G_{s_0}$ is injective by the minimality of $s_0$. Then $G_{s_0}^2=L_{2s_0}$ is also injective on ${\rm Ann}_{V}(L_{s_0})$, which contradicts the definition of $V$.

Since $S$ is simple and generated by $V$, we have following canonical surjective map
\begin{equation*}
\begin{aligned}
\pi:\
& {\rm Ind}(V)\rightarrow S \\
& 1\otimes v\mapsto v
\end{aligned}
\end{equation*}
for any $v\in V$. We remain to show $\pi$ is injective, i.e. ${\rm Ker}(\pi)=0$. Otherwise, we have ${\rm Ker}(\pi)\cap V\neq \{0\}$ by Lemma \ref{Lem (k,i)} since ${\rm Ker}(\pi)$ is a $\mathcal{R}$-submodule of ${\rm Ind}(V)$, which is absurd. Thus $\pi$ is a bijection and $V$ is a simple $\mathfrak{b}$-module as desired.
\end{proof}

As a direct consequence of Theorem \ref{Th equivalent}, we have the following classification of simple smooth modules over the Ramond algebra.

\begin{Corollary}\label{maincor}
 Any simple smooth module over the  Ramond algebra is either a simple highest weight module or isomorphic to the induced module ${\rm Ind}(V)$, where $V$ is a simple $\mathfrak{b}$-module satisfying the conditions in Lemma \ref{Lem G_jV}, that is, $V$ is a simple $\mathfrak{b}^{(t)}$-module for some $t\in\mathbb{Z}_+$.
\end{Corollary}


At the end of this section, we consider   simple modules over the quotient algebra $\mathfrak{b}^{(t)}$, and classify all simple $\mathfrak{b}^{(t)}$-modules for $t=0$ and $1$.

For $t=0$, $\mathfrak{b}^{(0)}=\mathfrak{b}^{(0)}_{\bar{0}}\oplus\mathfrak{b}^{(0)}_{\bar{1}}$ is a 3-dimensional solvable Lie superalgebra, where $\mathfrak{b}^{(0)}_{\bar{0}}=\mathbb{C}L_0\oplus\mathbb{C}c$, and $\mathfrak{b}^{(0)}_{\bar{1}}=\mathbb{C}G_0$. The subalgebra $\mathfrak{b}^{(0)}_{\bar{0}}$ is commutative and its simple modules are all one-dimensional.

Let $V=V_{\bar{0}}\oplus V_{\bar{1}}$ be a simple $\mathfrak{b}^{(0)}$-module. We can always assume $V_{\bar{0}}\neq 0$ up to parity-change.
If $V_{\bar{1}}=0$, then
$$G_0V=0=\left(L_0-\frac{1}{24}c\right)V,$$
yielding $V$ is a simple $\mathfrak{b}^{(0)}_{\bar{0}}$-module. Thus $V$ must be one-dimensional of the form $\mathbb{C}w$ with
\begin{equation}\label{equ b0 V=U}
\begin{aligned}
&L_0w=\lambda w,\ cw=lw,
\end{aligned}
\end{equation}
where $\lambda, l\in\mathbb{C}$ with $
\lambda=\frac{1}{24}l$.

If $V_{\bar{1}}\neq 0$. Take a nonzero $\mathfrak{b}^{(0)}_{\bar{0}}$-submodule $U$ of $V_{\bar{0}}$. It is easy to check that $U\oplus G_0U$ is a $\mathfrak{b}^{(0)}$-submodule of $V$, and  $U$ is actually a simple $\mathfrak{b}^{(0)}_{\bar{0}}$-module.  Thus $U=\mathbb{C}w$. Moreover, for any vector $v=aw+bG_0w\in V$, $a, b \in\mathbb{C}$,  we have
\begin{equation}\label{equ b0V}
\begin{aligned}
&L_0v=\lambda aw+\lambda bG_0w=\lambda v,\\
&cv=l aw+l bG_0w=l v,\\
&G_0v=aG_0w+bG_0^2w=b(\lambda-\frac{1}{24}l)w+aG_0w.\\
\end{aligned}
\end{equation}
Whence we have proved the following proposition.
\begin{Proposition}\label{Prop b0}
Any simple $\mathfrak{b}^{(0)}$-module $V$ is isomorphic to $\mathbb{C}w\oplus \mathbb{C}G_0w$ defined by (\ref{equ b0V}), where $\mathbb{C}w$ is a one-dimensional simple $\mathfrak{b}^{(0)}_{\bar{0}}$-module, up to parity-change. In particularly, $V$ is one-dimensional defined by (\ref{equ b0 V=U}), if $G_0w=0$.     \hfill$\square$
\end{Proposition}

Now we consider irreducible modules over $\mathfrak{b}^{(t)}$ for $t>0$. Note that  $\mathfrak{b}^{(t)}=\mathfrak{b}^{(t)}_{\bar{0}}\oplus\mathfrak{b}^{(t)}_{\bar{1}}$ is a $(2t+3)$-dimensional solvable Lie superalgebra, where $$\mathfrak{b}^{(t)}_{\bar{0}}=\mathbb{C}L_0\oplus\mathbb{C}L_1\oplus\cdots\oplus\mathbb{C}L_t\oplus\mathbb{C}c,\text{ and }\mathfrak{b}^{(t)}_{\bar{1}}=\mathbb{C}G_0\oplus\mathbb{C}G_1\oplus\cdots\oplus\mathbb{C}G_t.$$
According to Theorem \ref{Th equivalent}, we only need to consider irreducible $\mathfrak{b}^{(t)}$-modules on which  $L_t$ acts injectively.

For $t=1$, let $V=V_{\bar{0}}\oplus V_{\bar{1}}$ be a simple $\mathfrak{b}^{(1)}$-module with $V_{\bar{0}}\neq 0$ as above. We may assume that the central element
$c$ acts on $V$ as the scalar $l\in\mathbb C$.
Firstly suppose that $V_{\bar{1}}\neq 0$.
If $G_1V=0$, we have
$$L_1V=\frac{1}{2}(G_1G_0+G_0G_1)V=0,$$
which contradicts our assumption.
So $G_1V\neq 0$. From $G_1^2V=0$, there is always a nonzero $\mathfrak{b}^{(1)}_{\bar{0}}$-submodule $U$ of $G_1V_{\bar{1}}$ (or $G_1V_{\bar{0}}$). We see that
$G_1U=0.$
Then we can check that $U\oplus G_0U$ becomes a $\mathfrak{b}^{(1)}$-submodule of $V$, which implies that $V=U\oplus G_0U$. Furthermore, $U$ must be  a simple $\mathfrak{b}^{(1)}_{\bar{0}}$-module, and $G_0$ does not annihilate any nonzero vectors in $U$. Note that all simple $\mathfrak{b}^{(1)}_{\bar{0}}$-modules were classified by Block in \cite{B1981}. Now, $\mathfrak{b}^{(1)}$ acts on $V$ as
\begin{equation}\label{equ b1 GV}
\begin{aligned}
&xG_0u=G_0xu,\ \forall{x}\in \mathfrak{b}^{(1)}_{\bar{0}}, u\in U,\\
&G_0G_0u=(L_0-\frac{1}{24}l)u,\\
&G_1u=0, \,\,\,G_1G_0u=2L_1u.\\
\end{aligned}
\end{equation}
Note that $U$ and $G_0U$ are isomorphic irreducible $\mathfrak{b}^{(1)}_{\bar{0}}$-modules.

If $V_{\bar{1}}=0$, then $L_1V=0$, which contradicts our assumption.
Whence we have proved the following proposition.

\begin{Proposition} \label{Prop b1}
Any simple $\mathfrak{b}^{(1)}$-module $V$ is either a simple $\mathfrak{b}^{(0)}_{\bar{0}}$-module, or isomorphic to $U\oplus G_0U$ defined by (\ref{equ b1 GV}), where $U$ and $G_0U$ are simple $\mathfrak{b}^{(1)}_{\bar{0}}$-modules up to parity-change, and  $G_0$ does not annihilate any nonzero vectors in $U$. \hfill$\square$
\end{Proposition}

\section{Weak modules   for   vertex operator superalgebras}\label{Sec weakmod}
\textit{The Neveu-Schwarz algebra} $\mathcal{N}=\mathcal{N}_{\bar 0}\oplus\mathcal{N}_{\bar 1}$ is the Lie superalgebra with a designated basis
$\{L_m,G_{p}, c\mid m\in\Z,p\in\frac{1}{2}+\Z\}$, where $\mathcal{N}_{\bar 0}=\mathrm{span}\{L_m,c\mid m\in\Z\}$, $\mathcal{N}_{\bar 1}=\mathrm{span}\{G_p\mid p\in\frac{1}{2}+\Z\}$, and
the Lie super-bracket is given by
\begin{equation*}\label{def1.1}
	\aligned
	&[L_m,L_n]= (m-n)L_{m+n}+\frac{m^{3}-m}{12}\delta_{m+n,0}c,\\&
	[G_p,G_q]= 2L_{p+q}+\frac{4p^{2}-1}{12}\delta_{p+q,0}c,\\&
	[L_m,G_p]= \left(\frac{m}{2}-p\right)G_{m+p},\ \  [\mathcal{N},c]=0,
	\endaligned
\end{equation*}
for $m,n\in\Z$, $p,q\in\frac{1}{2}+\Z$.
Note that $\mathcal{N}$  is isomorphic to a subalgebra of $\R$, which is spanned by
$\{L_{m}\mid m\in2\Z\}\cup\{G_p\mid p\in2\Z+1\}\cup\{c\}$.
It is clear that $\mathcal{N}$  has a $\frac{1}{2}\Z$-grading given by the adjoint action
of $L_0$.  Then $\mathcal{N}$ has the following triangular decomposition:
$$\mathcal{N}=\mathcal{N}_+\oplus\mathcal{N}_0\oplus\mathcal{N}_-,$$ where $\mathcal{N}_+=\mathrm{span}\{L_m,G_p\mid m,p>0\}$, $\mathcal{N}_-=\mathrm{span}\{L_m,G_p\mid m,p<0\}$ and $\mathcal{N}_0=\C\{L_0,c\}$.

Set
\begin{eqnarray}\label{LTG511}
L(z)=\sum_{m\in\Z}L_mz^{-m-2},\ \ G(z)=\sum_{n\in\Z}G_{n+\frac{1}{2}}z^{-n-2}.
\end{eqnarray}
From Section $4.2$ of \cite{L}, we have
\begin{eqnarray*}
	&&[L(z_1),L(z_2)]=z_1^{-1}\delta\left(\frac{z_2}{z_1}\right)\frac{d}{d z_2}(L(z_2))
	+2\frac{\partial}{\partial z_2}(z_1^{-1}\delta\left(\frac{z_2}{z_1}\right))L(z_2)
	+\frac{c}{12}(\frac{\partial}{\partial z_2})^3z_1^{-1}\delta\left(\frac{z_2}{z_1}\right),\\
	&&[L(z_1),G(z_2)]=z_1^{-1}\delta\left(\frac{z_2}{z_1}\right)\frac{\partial}{\partial z_2}(G(z_2))
	+\frac{3}{2}\big(\frac{\partial}{\partial z_2}z_1^{-1}\delta\left(\frac{z_2}{z_1}\right)\big)G(z_2),\\
	&&[G(z_1),G(z_2)]=2z_1^{-1}\delta\left(\frac{z_2}{z_1}\right)L(z_2)
	+\frac{c}{3}(\frac{\partial}{\partial z_2})^2z_1^{-1}\delta\left(\frac{z_2}{z_1}\right).
\end{eqnarray*}

For any $h,l\in\C$, let $W(h,l)$ be the Verma module for $\mathcal{N}$ with  highest weight $(h,l)$ with
 a highest weight vector $v$. Set
\begin{eqnarray}\label{5342}
\bar{W}(0,l)=W(0,l)/\langle G_{-\frac{1}{2}}\mathbf{1}\rangle,
\end{eqnarray}
where  $\langle G_{-\frac{1}{2}}v\rangle$
denotes the $\mathcal{N}$-submodule generated by $G_{-\frac{1}{2}}\mathbf{1}$.

Set $\mathbf{1}=v+\langle G_{-\frac{1}{2}}v\rangle\in \bar{W}(0,l)$ and then set
\begin{eqnarray}
\omega=L_{-2}{\bf 1},\ \tau=G_{-\frac{3}{2}}{\bf 1}\in \bar{W}(0,l).
\end{eqnarray}
Then (see \cite{KW,L}) $\bar{W}(0,l)$ admits a vertex operator superalgebra structure which is uniquely determined by
the condition that $\mathbf{1}$ is the vacuum vector and
\begin{eqnarray}
Y(\omega,z)=L(z),\quad Y(\tau,z)=G(z).
\end{eqnarray}

Assume that $\V=\V^{(0)}\oplus  \V^{(1)}$ is  a vertex superalgebra. Define  the following linear map
\begin{eqnarray*}
\psi:\quad   \V&\longrightarrow& \V \\
a+b&\longmapsto&a-b
\end{eqnarray*}
for $a\in \V^{(0)}, b\in \V^{(1)}$. Then $\psi$  is an automorphism of $\V$,
 called the {\em canonical automorphism} (see \cite{FFF}). Furthermore, $\mathrm{Aut}(\bar W(0,c))=\langle \psi\rangle=\Z_2$.

The following results can be found in \cite{L,L1}:

\begin{Lemma}\label{lemm51}
Let $l\in\C$ and define $\bar{W}(0,l)$ as in \eqref{5342}.
Then any weak $\psi$-twisted $\bar W(0,l)$-module $(W,Y_W)$  is a smooth $\R$-module of central charge $l$
with
$$L(z)=Y_W(\omega,z),\quad G(z)=Y_W(\tau,z). $$
On the other hand, for any smooth $\R$-module $W$ of central charge $l$, there exists a weak $\psi$-twisted $\bar W(0,l)$-module
structure $Y_W(\cdot,z)$ on $W$, uniquely determined by
$$Y_W(\omega,z)=L(z),\quad Y_W(\tau,z)=G(z). $$
\end{Lemma}

Combining Corollary \ref{maincor} and  Lemma \ref{lemm51}, we immediately have the following results:

\begin{Proposition}\label{5.2}
Let $l\in\C$. Then all simple weak $\psi$-twisted $\bar W(0,l)$-modules are precisely irreducible smooth $\R$-modules given in Corollary \ref{maincor}.
\end{Proposition}

\section{Examples}\label{Sec Examples}
In this section, we give some examples of simple smooth modules over the Ramond algebra, including the highest weight modules and Whittaker modules. That is, we construct some induced modules satisfying the conditions in Theorem \ref{Th simple}, so that the theorem can be applied.


\subsection{Highest weight modules}
Recall that
$$\mathcal{R}^{(0,1)}:={\bigoplus\limits_{i\geq 0}{\mathbb{C}L_i}}
\oplus {\bigoplus\limits_{j\geq 1}{\mathbb{C}G_j}} \oplus {\mathbb{C}c}.$$
Let $\mathbb{C}v$ be a one-dimensional $\mathcal{R}^{(0, 1)}$-module defined by
$$L_0v=\lambda v,\  cv=lv,\ \mathcal{R}_+v=0,$$
for $\lambda, l\in\mathbb{C}$. The $\mathfrak{b}$-module $V(\lambda,l)$ is given as
\begin{equation*}
V(\lambda,l):=
\begin{cases}
  \mathbb{C}v & \mbox{if } \lambda=\frac{1}{24}c, \\
  {\rm Ind}_{\mathcal{R}^{(0, 1)}}^{\mathfrak{b}}\mathbb{C}v & \mbox{otherwise}.
\end{cases}
\end{equation*}
Then the Verma module $M(\lambda, l)$ over the Ramond algebra (see \cite{IK2003-1}) is defined by
$$M(\lambda, l)={\rm Ind}_{\mathfrak{b}}^{\mathcal{R}}V(\lambda,l).$$
The module $M(\lambda, l)$ has the unique simple quotient $L(\lambda, l)$, which is the unique simple highest weight module with the highest weight $(\lambda, l)$, up to isomorphism. These simple modules correspond to the highest weight modules in Theorem \ref{Th equivalent}.

\subsection{Whittaker modules}
Let
$$\mathfrak{p}=\bigoplus\limits_{m\geq 1}\mathbb{C}L_m\oplus\bigoplus\limits_{n\geq 2}\mathbb{C}G_n$$
and $\phi:\mathfrak{p}\rightarrow \mathbb{C}$ be a Lie superalgebra homomorphism. Then $\phi(L_m)=\phi(G_n)=0$ for all $m\geq3$ and $n\geq 2$. All finite dimensional simple modules over $\mathfrak{p}$ have been classified in \cite{LPX2019}.
Let $\mathbb{C}w$ be a one-dimensional $(\mathfrak{p}\oplus\mathbb{C}c)$-module with
$$xw=\phi(x)w,\ cw=lw,$$
for $x\in\mathfrak{p}$ and $l\in\mathbb{C}$.
Then the Whittaker module $W(\phi, l)$ over the Ramond algebra is defined by
$$W(\phi, l)=U(\mathcal{R})\otimes_{U(\mathfrak{p}\oplus\mathbb{C}c)}\mathbb{C}w.$$
By \cite{LPX2019}, $W(\phi, l)$ is simple if and only if $\phi(L_2)\neq 0$.

Let $\phi:\mathfrak{p}\rightarrow \mathbb{C}$ be the nontrivial Lie superalgebra homomorphism with $\phi(L_2)\neq 0$, and $A_{\phi}=\mathbb{C}w\oplus\mathbb{C}u$ be a two-dimensional vector space with
$$xw=\phi(x)w,\ G_1w=u,\ cw=lw$$
for all $x\in\mathfrak{p}$. Then $A_{\phi}$ is a simple $(\mathcal{R}_+\oplus\mathbb{C}c)$-module. Consider the induced module
$$V_{\phi}=U(\mathfrak{b})\otimes_{U(\mathcal{R}_+\oplus\mathbb{C}c)}A_{\phi}.$$
It is straightforward to check that $V_{\phi}$ is a simple $\mathfrak{b}$-module if $\phi(L_2)\neq 0$. The corresponding simple induced $\mathcal{R}$-module ${\rm Ind}(V_{\phi})$ is obtained by Theorem \ref{Th simple}, which is exactly the Whittaker module $W(\phi, l)$.

\subsection{High order Whittaker modules}
For $t\in\mathbb{N}$, let
$$\mathfrak{p}^{(t)}=\bigoplus\limits_{m> t}\mathbb{C}L_m\oplus\bigoplus\limits_{n> t+1}\mathbb{C}G_n.$$
Clearly, $\mathfrak{p}^{(0)}=\mathfrak{p}$. Let $\phi_t:\mathfrak{p}^{(t)}\rightarrow \mathbb{C}$ be a Lie superalgebra homomorphism for $t\in\mathbb{Z}_+$. Then $\phi_t(L_m)=\phi_t(G_n)=0$ for $m\geq 2t+3$ and $n\geq t+2$. Let $\mathbb{C}w$ be a one-dimensional $(\mathfrak{p}^{(t)}\oplus\mathbb{C}c)$-module with
$$xw=\phi_t(x)w,\ cw=lw,$$
for $x\in\mathfrak{p}^{(t)}$ and $l\in\mathbb{C}$.
The high order Whittaker module $W(\phi_t, l)$ over the Ramond algebra is given by
$$W(\phi_t, l)=U(\mathcal{R})\otimes_{U(\mathfrak{p}^{(t)}\oplus\mathbb{C}c)}\mathbb{C}w.$$

Let $A_{\phi_t}=\mathbb{C}w_t\oplus\mathbb{C}u_t$ be a two-dimensional vector space with
$$xw_t=\phi_t(x)w_t,\ G_{t+1}w_t=u_t,\ cw_t=lw_t$$
for all $x\in\mathfrak{p}^{(t)}$. Similar to \cite{C2023} and \cite{LPX2019}, we deduce that $A_{\phi_t}$ is a simple $(\mathfrak{m}^{(t)}\oplus\mathbb{C}c)$-module if and only if $\phi_t(L_{2t+2})\neq 0$. (In fact, $A_{\phi_t}$ has a submodule spanned by $u_t$ if $\phi_t(L_{2t+2})=0$.)
Consider the induced module
$$V_{\phi_t}=U(\mathfrak{b})\otimes_{U(\mathfrak{m}^{(t)}\oplus\mathbb{C}c)}A_{\phi_t}.$$
It is straightforward to check $V_{\phi_t}$ is a simple $\mathfrak{b}$-module if $\phi_t(L_{2t+2})\neq 0$. The corresponding simple induced $\mathcal{R}$-module ${\rm Ind}(V_{\phi_t})$ is provided by Theorem \ref{Th simple}, which is exactly the high order Whittaker module $W(\phi_t, l)$.

\section*{Acknowledgments}
This work was carried out during the first author's visit to Wilfrid Laurier University. She gratefully acknowledges the hospitality of professor Kaiming Zhao and Wilfrid Laurier University, and thanks the help of professor Ran Shen. The authors also thank professor Haisheng Li for revision on applications to vertex operator superalgebras in section 5, and are grateful to professors Chongying Dong, Li Ren and Bin Wang for helpful discussions. We thank the referees for nice and helpful suggestions.

\end{document}